\documentclass[a4paper,12pt]{article}
    \usepackage[top=2.5cm,bottom=2.5cm,left=2.5cm,right=2.5cm]{geometry}
    \usepackage{cite, amsmath, amssymb}
    \usepackage[margin=1cm,%
                font=small,%
                format=hang,%
                labelsep=period,%
                labelfont=bf]{caption}
\usepackage{amsthm}

\numberwithin{equation}{section}
\newtheorem{thm}{Theorem}[section]
\newtheorem{prop}[thm]{Proposition}
\newtheorem{cor}[thm]{Corollary}
\newtheorem{lem}[thm]{Lemma}

\title{\vspace*{2cm}Graphs of Maximal Energy with Fixed Maximal Degree}

\date{}
\author {Octavio Arizmendi \\ Centro de Investigaci\'on en Matem\'aticas \\ Guanajauto, Mexico. Email: octavius@cimat.mx \\ 
\\ Jorge Fernandez Hidalgo \\ Universidad Nacional Aut\'onoma de Mexico \\ Mexico City, Mexico. Email: jorgefernandez@ciencias.unam.mx}

\begin{document}
\maketitle
\begin{center}(accepted in MATCH) \end{center}
\abstract{We give a bound for the graph energy with given maximal degree in terms of the second and fourth moments of a graph. In the case in which the graph is $d$-regular we obtain the bound that is given in Van Dam, E. et al. (2014). through elementary methods.
}
 \baselineskip=0.30in

\section{ Introduction and Statement of Results}

In this paper we study the energy of a graph as defined by Gutman \cite{Gut}. For a graph $G$ on $n$ vertexes with adjacency eigenvalues $\lambda_1,\dots,\lambda_n$, the energy of the graph $G$ is given by sum of the absolute value of its adjacency eigenvalues,
$$\sum_{i=1}^n |\lambda_i|.$$

Several results on bounds for the energy of a graph have been considered in the theory. e.g. \cite{KM,McClelland, Zhou}. A tool to give general bounds is given by the so-called spectral moments of a graph \cite{GGQ,EB, NiK2,Rada2,RaTi,ZhoGutPe}. The $k$-th spectral moment of a graph is given by
$$M_k(G)= \sum\limits_{i=1}^n \lambda_i^ k.$$

The usefulness of these quantities comes from the fact that, when $k$ is an integer, $M_k$ has a combinatorial interpretation: $M_k(G)$ is the trace of the $k$'th power of the adjacency matrix of $G$, and consequently, is given by the number of closed walks of length $i$ in $G$.

Here we present a simple  but effective method to bound the energy of a graph in terms of its spectral moments which we describe in Lemma \ref{polibound}.

A direct application of  this method provides upper and lower bounds for the energy of a graph with $n$ vertexes, and maximum degree $\Delta$, in terms of the second and fourth moments.

To state this result, we shall use the notation
 \begin{align} \label{moments}
A = \frac{M_4}{\Delta^3}, \quad B = \frac{M_2}{\Delta}, \quad C = \Delta n.
\end{align}

Our main theorem is as follows.
\begin {thm} \label{T1}
Let $G$ be a connected graph with at least $2$ vertexes,
\begin{equation} \label{main inequality}
E(G) \leq -\frac{B^2 + B\sqrt{-A+B}\sqrt{-B+C} - C( A + \sqrt{-A+B}\sqrt{-B+C})}{A-2B+C}
\end{equation}
with equality if and only if $G$ is a complete graph $K_n$, a strongly regular graph with $\lambda=\mu$ or the incidence graph of a symmetric $2-(v,k,\lambda)$ design.
\end{thm}

For regular graphs one can see that \eqref{main inequality} is decreasing in $A$ (see Lemma \ref{decrece}), and thus the above theorem subsumes the main theorem of  Van Dam et al. \cite{Dam}.
\begin{thm}[\cite{Dam}] \label{T2}
For a regular graph $G$ one has
\begin{equation*}
E(G) \leq  n \frac{d+(d^2 -d)\sqrt{d-1}}{d^2 -d +1}
\end{equation*}
with equality if and only if $G$ is the incidence graph of a symmetric $2-(v,d,1)$ design.
\end{thm}

We must mention that the  above result was derived originally by very different methods from the ones in this paper and we do not know if the methods from \cite{Dam}, may be applied to show our main theorem.

Apart from the introduction, the paper has two more sections. Section 2 gives the basic preliminaries needed for this paper and the proofs of the main results are given in Section 3.

\section{Preliminaries on graphs}

\subsection{Basic definitions}

Throughout this paper, we shall work exclusively with simple finite graphs. We shall denote the number of vertexes of the graph $G$ by $n$ and the number of edges by $m$. We will denote the vertexes of a graph $G$ using $v_1,\dots,v_n$. We say vertex $v_i$ is a neighbour of $v_j$ if $v_i$ and $v_j$ are adjacent (in other words if $(v_i,v_j)$ is an edge of $G$). We define the degree of a vertex $v_i$ as its number of neighbours and we denote this quantity by $d_i$. The maximum degree over all vertexes of the graph will be denoted by $\Delta$, while the minimum degree will be denoted by $\delta$. If $\Delta=\delta=d$ we say that $G$ is a $d$-regular graph.

Given a graph $G$, the adjacency matrix of $G$ is the $n\times n$ symmetric matrix $A$ such that $A_{ij}$ is $1$ if $i$ and $j$ are adjacent and $A_{i,j}$ is $0$ otherwise. Being a symmetric matrix, $A$ has $n$ real eigenvalues, counted with multiplicity, which are called the adjacency eigenvalues of the graph $G$, we shall sometimes simply refer to them as the eigenvalues of $G$.

\subsection{Spectral moments}

The spectral moments of the graph $G$ are the quantities $M_k=Tr(A^k)$.  Since $A$ is normal then $Tr(A^k)=\sum\limits_{i=1}^n \lambda_i^ k$.
Let us note that the $0$-th, first and second moments are given by 
\begin{equation*}\label{m0} M_0=n,~ M_1 = 0, \text{ and } M_2=0,\end{equation*}
 let us denote the number of $4$-cycles in $G$, by $Q$ and use $Z$ for the Zagreb index given by  $Z=\sum\limits_{i=1}^ n d_i^ 2$, then
\begin{equation*}\label{m4}
M_4 = 2Z-2m +8Q.
\end{equation*}
Hence, with the notation of the introdution, we have, 
 \begin{align} \label{AQ}
A = \frac{2Z-2m+8Q}{\Delta^3}, \quad B = \frac{2m}{\Delta}, \quad C = \Delta n.
\end{align}

\subsection{Graphs with few eigenvalues}

Finally,  we describe the graphs which satisfy the equality in Theorems \ref{T1} and \ref{T2}.

We say that a graph is strongly connected with parameters $\lambda,\mu$ if it is regular and it satisfies: the following properties
\begin{itemize}
\item Every two adjacent vertexes have $\lambda$ common neighbours.
\item Every two non adjacent vertexes have $\mu$ common neighbours.
\end{itemize}

The spectrum of a strongly connected graph with $n$ vertexes and common degree $d$ is:

\begin{itemize}
\item $d$ with multiplicity $1$.
\item $\frac{1}{2}[(\lambda-\mu)+\sqrt{(\lambda-\mu)^2+4(d-\mu)}]$ with multiplicity $\frac{1}{2}[(v-1)-\frac{2k+(v-1)(\lambda-\mu)}{\sqrt{(\lambda-\mu)^2+4(d-\mu)}}],$
\item $\frac{1}{2}[(\lambda-\mu)-\sqrt{(\lambda-\mu)^2+4(d-\mu)}]$ with multiplicity $\frac{1}{2}[(v-1)+\frac{2k+(v-1)(\lambda-\mu)}{\sqrt{(\lambda-\mu)^2+4(d-\mu)}}].$
\end{itemize}

We refer to a symmetric $2-(v,k,\lambda)$ design as a set $X$ with $v$ elements, along with a family $F$ consisting of $v$ subsets of $X$ (called blocks) such that each block has $k$ elements and for any two points $x$ and $y$ in $X$ there are exactly $\lambda$ blocks that contain both.

We define the incidence graph of a symmetric $2-(v,k,\lambda)$ design as the graph with vertexes $F\cup X$ such that there is an edge between a block $B\in F$ and a point $x\in X$ if and only if $x\in B$.

The spectrum of the incidence graph of a symmetric $2-(v,k,\lambda)$ design is:
\begin{itemize}
\item $\pm k$, each with multiplicity $1$.
\item $\pm \sqrt{k-\lambda}$, each with multiplicity $v-1$.
\end{itemize}

\section{Bounds for the energy of a graph using moments}

\subsection{A general bound for the energy of a graph}

The main idea comes from the following observation. Let $P(x)=a_mx^m + \dots + a_0$ be a polynomial such that  $P(x)\geq |x|$ for all $x\in [-\rho,\rho]$, then $$E(G)= \sum\limits_{i=1}^n |\lambda_i| \leq \sum\limits_{i=1}^n P(\lambda_i) = \sum\limits_{i=0}^m a_m M_m.$$ Here we emphasize that in order to find useful bounds for $E(G)$ we only need to take into consideration $x\in [-\rho,\rho]$ and not in all $\mathbb{R}$. Since obtaining $\rho$ directly from combinatorial properties of the graph is, in general hard we use instead $\Delta$, obtaining the following lemma, which includes analogous lower bounds for $E(G)$.
\begin{lem}\label{polibound} Let G be a graph with maximum degree $\Delta$.

1)  For any polynomial $P(x)=a_mx^m + \dots + a_0$ a polynomial such that  $P(x)\geq |x|$ on $x\in [-\Delta,\Delta]$ we have
 \begin{equation}\label{equa1}E(G)\leq \sum\limits_{i=0}^m a_m M_m.\end{equation}
 
2) For any polynomial $Q(x)=a_mx^m + \dots + a_0$ be a polynomial $Q(x)\leq |x|$, on $x\in [-\Delta,\Delta]$  we have  \begin{equation}\label{equa2} E(G)\geq \sum\limits_{i=0}^m a_m M_m.\end{equation}
 Moreover, equality in \eqref{equa1} (resp. \eqref{equa2}) occurs if and only if $P(\lambda_i)=|\lambda_i|$ (resp.  $Q(\lambda_i)=|\lambda_i|$) for all $i$.
\end{lem}

The strength of the previous lemma comes  from  the freedom of choosing the above polynomials, and using the information of the class of graphs studied is very helpful as we will see in the next sections. 

\subsection{Upper bounds in terms of second and fourth moments}

Let us consider the case where $P$ is a polynomial of degree $4$. Since the absolute value is an even function considering $x$ or $x^3$ in the polynomial will produce one of the sides to tilt which will worsen our approximation. Thus we assume that $P$ is even. Then we are led consider a polynomial of the form $P(x)=ax^4+bx^2+c.$

For this polynomial $P(x)= ax^4 + bx^2 + c$ we have
 \begin{equation}\label{momentspoli}
Tr(P(A))=\sum\limits_{i=1}^nP(\lambda_i) =aM_4 + bM_2 + cM_0= 2aZ+8aQ + 2(b-a)m + cn.
\end{equation}

Moreover, we see that \begin{eqnarray}\label{objective} Tr(P(A)) &=& aM_4 + bM_2 + cM_0 = a( 2(\sum\limits_{i=0}^n d_i^2) -2m +8Q)  + bm + cn  \\&=& a8Q + (b-2)m + (c+ad^2)n = a8Q + \left(c + ad^2 + \frac{bd}{2} -d\right)n \nonumber.\end{eqnarray}

We note for further reference that when $G$ is $d$-regular we have $2m=dn$ and $Z=nd^2$ which gives us
\begin{equation}\label{momentspolireg}
\sum\limits_{i=1}^nP(\lambda_i) = 2and^2 +8aQ + (b-a)nd + cn.
\end{equation}

Now our goal is to choose $a,b$ and $c$ which minimize the left-hand side of (\ref{objective}) restricted to the condition that $P(x)\geq x$ for $\Delta \geq x\geq0$. For this we seek for polynomials that are tangent to $f(x)=x$.

\begin{lem}
Given $0<r<1$ there is a unique polynomial $P_{r}=ax^4+bx^2+c$ such that $P_r(r)=r,P_r(1)=1$  and $P_r(x)\geq x$ for $x\in [0,1]$.
\end{lem}
\begin{proof}

We shall call the desired polynomial $P_r=ax^4+bx^2+c$.
Since $P_r$ must be tangent to the absolute value function at $r$ the following equations must be satisfied
\begin{eqnarray*} 
P_r(r)&=&ar^4+br^2+c = r, \\
P_r(1)&=&a+b+c = 1, \\
P_r'(r)&=& 4ar^3 + 2br = 1.
\end{eqnarray*} 
The solution to this system of equations is unique and gives us
\begin{equation}a = -\frac{1}{2r(r+1)^2}, \quad b = \frac{3r^ 2 +2r + 1}{2r(r+1)^ 2}, \quad c= \frac{r^ 2(2r + 1)}{2r(r+1)^ 2}. \end{equation}

First we show that $P_r(x)> x$ for all $0 < x < 1$ with $x\neq r$.  Consider for this, the function $Q(x)=P(x)-x$. Its second derivative is given by  $12ax^2 - 2b$, which is decreasing and thus $Q(x)$ is convex on $(0,r_0)$ and concave on $(r_0,1)$, where $r_0=\sqrt{\frac{3r^2 +2r + 1}{6}}$ is the unique solution to $Q''(x)=0$. Since $r\in(0,1)$,  then $r<r_0<1$. Since in the interval $[0,r_0]$ the function is convex it reaches its minimum at the point such that $Q'(x)=0$. This point is $r$ and we have $Q(r)=0$. In the interval $[r_0,1]$ the function is concave and therefore reaches its minimum at one of $r_0$ and $1$. We have $Q(r_0)> Q(r)=0$ and we have $Q(1)=0$. So indeed $P_r$ is as desired.

 
\end{proof}

Thus from Lemma \ref{polibound} we obtain the following bound.
	
\begin{thm}\label{acotacion}
Let $r\in(0,1)$. For $G$ a graph with maximum degree $\Delta$ the energy of $G$ is bounded by  
$$E(G)\leq  \frac{-1}{2r(r+1)^2}\frac{2Z-2m+8Q}{\Delta^3}  +(\frac{3r^2+2r+1}{2r(r+1)^2}) \frac{2m}{\Delta}  + n\Delta \frac{r^2(2r+1)}{2r(r+1)^2}$$
\end{thm}

\begin{proof}

Let $P(x)$ be as in Lemma 2.2. Then $P_r(x)\geq |x|$ for $x\in[-1,1]$. Now let  $P_{r,\Delta}(x) = \Delta P_{r}(x/R)$, be the dilation by $\Delta$ of the polynomial $P_{r}$.  Then $P_{r,\Delta}(x)\geq |x|$ for $x\in[-\Delta,\Delta]$. 
By part 1) of Lemma 2.1 we conclude.
\end{proof}

For the following we recall the notation
\begin{align*}
A = \frac{2Z-2m+8Q}{\Delta^3}, \quad B = \frac{2m}{\Delta}, \quad C = \Delta n
\end{align*}

We want to minimize for $r$. For this, we claim that $A \leq B\leq C$ and that $0\leq \frac{B-A}{C-B}\leq1$. To show that $B\leq C$ we simply notice $2m\leq \Delta n \Delta^2 n$, and to prove $ A \leq B$ we see it is equivalent to $M_4\leq 2m \Delta^2$. Now,  notice that for every oriented edge $u,v$ there can be at most $\Delta^2$ closed walks of length $4$ starting with $u,v$.

By elementary calculus one sees that subject to the restriction $A \leq B$ and $B\leq C$. The minimum value of the function $\frac{-1}{2r(r+1)^2}A  +(\frac{3r^2+2r+1}{2r(r+1)^2})B  + \frac{r^2(2r+1)}{2r(r+1)^2}C$ for $r\in (0,1)$ is reached at  $\frac{\sqrt{B-A}}{\sqrt{C-B}}$, giving the value
\begin{equation}\label{minimo}
-\frac{B^2 + B\sqrt{B-A}\sqrt{C-B} - C( A + \sqrt{B-A}\sqrt{C-B})}{A-2B+C}.
\end{equation}
 
Thus we arrive to the following theorem.
\begin {thm}
For any graph $G$ one has
\begin{equation*}
E(G) \leq -\frac{B^2 + B\sqrt{-A+B}\sqrt{-B+C} - C( A + \sqrt{-A+B}\sqrt{-B+C})}{A-2B+C}
\end{equation*}
with equality if and only if the spectrum of the graph is contained in $\{\pm \frac{\sqrt{B-A}}{\sqrt{C-B}}, \pm \Delta\}$
\end{thm}

\begin{proof} 
The inequality follows from the considerations above. In order to have equality, the polynomial $P_{\frac{\sqrt{B-A}}{\sqrt{C-B}},\Delta}(x)$ and the function $Abs(x):=|x|$ should coincide for all eigenvalues of $G$. This is only possible for the set $\{\pm \frac{\sqrt{B-A}}{\sqrt{C-B}}, \pm \Delta\}.$
\end{proof}
The following propositon characterizes the graphs for which equality is reached and thus completes the main theorem of this paper.

\begin{prop} The only connected graphs with at least $2$ vertexes for which equality is attained are the complete graphs $K_n$, strongly regular graphs with $\lambda=\mu$ and incidence graphs of symmetric $2-(v,k,\lambda)$ designs.
\end{prop}

\begin{proof}
Equality occurs if and only if the spectra of $G$ is contained in $\{\pm \frac{\sqrt{B-A}}{\sqrt{C-B}}, \pm \Delta\}$.

If $\Delta$ is not an eigenvalue of $G$ then $-\Delta$ is not an eigenvalue either, so $G$ has exactly two eigenvalues. This implies $G$ is complete, however complete graphs have $\Delta$ as an eigenvalue. We conclude $\Delta$ is an eigenvalue, meaning $G$ is regular.

If $-\Delta$ is an eigenvalue then $G$ is bipartite and therefore has symmetric spectra. If $G$ has exactly $2$ eigenvalues then $G$ is complete and bipartite (So $G=K_2$). If $G$ has exactly $4$ eigenvalues then $G$ is a regular bipartite graph with exactly $4$ eigenvalues and must therefore be the incidence graph of a symmetric $2-(v,k,\lambda)$ design, as shown in \cite{CveDo}.

If $-\Delta$ is not an eigenvalue then $G$ must have exactly $3$ eigenvalues or exactly $2$ eigenvalues. The only connected regular graphs with exactly $3$ eigenvalues are the strongly regular graphs, as is show in \cite{God}. Moreover  $\lambda=\mu$ is required so that one eigenvalue is the negative of the other. Finally, if $G$ has only two eigenvalues it must be a complete graph.

It should be clear that equality is indeed achieved in all these cases, as a polynomial $P_{r,\Delta}(x)$ can be found such that $P(x)=|x|$ for all eigenvalues $x$ of $G$.

\end{proof}

To minimize  \eqref{minimo} we notice the following.
\begin{lem}\label{decrece} 
The expression given in \eqref{minimo} is decreasing in $A$. 
\end{lem}

\begin{proof}

Note that the formula  $\frac{-1}{2r(r+1)^2}A  +(\frac{3r^2+2r+1}{2r(r+1)^2})B  + \frac{r^2(2r+1)}{2r(r+1)^2}C$ is decreasing in $A$ for every value of $r\in (0,1)$. The result follows since \eqref{minimo} is the the minimum of the expression over all $r\in (0,1)$ and thus should also be decreasing on $A$.
\end{proof}

Thus, from \eqref{AQ}, one should expect that having quadrangles should decrease the energy. 

\subsection{$d$-regular graphs}

Finally, we may recover the result in \cite{Dam}. Indeed, for the case of regular graphs we have 
\begin{align*}
A = \frac{2nd^2+nd +8Q}{d^3},\quad
B = n, \quad
C = nd.
\end{align*}
Lemma \ref{decrece} shows that, for $n$ and $d$ are fixed,  \eqref{minimo} is decreasing in $Q$ and then $Q=0$ gives a bound for all regular graphs. In this case the minimum value simplifies and is reached at $\frac{1}{d}\sqrt{d-1}$. 
\begin{cor}
For a regular graph $G$ one has
\begin{equation*}
E(G) \leq  n \frac{d+(d^2 -d)\sqrt{d-1}}{d^2 -d +1}
\end{equation*}
with equality if and only if $G$ is the incidence graph of a symmetric $2-(v,d,1)$ design.
\end{cor}

 \section*{Acknowledgment}

 O. Arizmendi was supported by a CONACYT Grant No. 222668 and by the  European Union's Horizon 2020 research and innovation programme under the Marie Sk\l{}odowska-Curie grant agreement 734922, during the writing of this paper.

\end{document}